\documentclass[12pt,leqno]{amsart}
\usepackage{amssymb,amsthm,amsmath,latexsym,wasysym}
\usepackage{soul}
\usepackage[all]{xy}
\usepackage{xcolor}

\newtheorem{theorem}{\sc Theorem}[section]
\newtheorem{thm}[theorem]{\sc Theorem}
\newtheorem{lem}[theorem]{\sc Lemma}
\newtheorem{prop}[theorem]{\sc Proposition}
\newtheorem{cor}[theorem]{\sc Corollary}

\newtheorem{rem}[theorem]{\sc Remark}

\newtheorem*{con1'}{Conjecture 1'}

\title[$n$-fold tensor product of groups]{Finiteness conditions for the $n$-fold tensor product of groups}

\author[Bastos]{Raimundo Bastos}
\author[Ortega]{Guilherme Ortega}
\address{Departamento de Matem\'atica, Universidade de Bras\'ilia,
Brasilia-DF, 70910-900 Brazil}

\email{(Bastos) bastos@mat.unb.br; (Ortega) ortega@mat.unb.br}

\begin{document}
\maketitle
\begin{abstract}
Let $G$ be a finitely generated group. We prove that the $n$-fold tensor product $G^{\otimes n}$ is finite (resp. polycyclic) if and only $G$ is finite (resp. polycyclic). Further, assuming that $G$ is finitely presented, we show that $G^{\otimes n}$ is finitely presented if and only if $\gamma_n(G)$ is finitely presented. We also examine some finiteness conditions for the non-abelian tensor product of groups. 
\end{abstract}

\maketitle

\section{Introduction}

Let $G$ and $H$ be groups each of which acts upon the other (on the right),
\[
G\times H \rightarrow G, \; (g,h) \mapsto g^h; \; \; H\times G \rightarrow
H, \; (h,g) \mapsto h^g
\]
and on itself by conjugation, in such a way that for all $g,g_1 \in G$ and
$h,h_1 \in H$,
\begin{equation}   \label{eq:0}
g^{\left( h^{g_1} \right) } = \left( \left( g^{g^{-1}_1}  \right) ^h \right) ^{g_1} \; \; \mbox{and} \; \; h^{\left( g^{h_1}\right) } =
\left( \left( h^{h_1^{-1}} \right) ^g \right) ^{h_1}.
\end{equation}
In this situation we say that $G$ and $H$ act {\em compatibly} on each other. 

Let $G$ and $H$ be groups that act compatibly on each other. The non-abelian tensor product $G \otimes H$ was introduced by Brown and Loday \cite{BL} following works of Miller \cite{Miller}, Dennis \cite{Dennis} and Lue \cite{Lue}. It is defined to be the group generated by all symbols $\; \, g\otimes h, \; g \in G, \; h\in H$, subject to the relations
\[
gg_1 \otimes h = ( g^{g_1}\otimes h^{g_1}) (g_1\otimes h) \quad
\mbox{and} \quad g\otimes hh_1 = (g\otimes h_1)( g^{h_1} \otimes
h^{h_1})
\]
for all $g,g_1 \in G$, $h,h_1 \in H$. When $G = H$ and all actions are conjugations,  $G \otimes G$ is the non-abelian tensor square. Now, there is a well defined action
of $G$ on $G \otimes G$ by $$(g_1 \otimes g_2)^{g_3} = g_1^{g_3} \otimes g_2^{g_3},$$ where $g_i \in G$. Moreover, there is a natural action of $G \otimes G$ on $G$ given by $$g_1^{g_2 \otimes g_3} = g_1^{[g_2,g_3]}.  $$ With these (compatible) actions we write $G^{\otimes 3}$ to denote the non-abelian tensor product $(G \otimes G) \otimes G$. Furthermore, for any $n \geqslant 3$, we can inductively define the $n$-fold tensor product, denoted by $G^{\otimes n}$, by considering the actions of $G$ and $G^{\otimes  n-1}$ on each other defined by $$(((g_1 \otimes g_2) \otimes \ldots \otimes g_{n-2}) \otimes g_{n-1})^{g_n} = ((g_1^{g_n} \otimes g_2^{g_n}) \otimes \ldots \otimes g_{n-2}^{g_n})\otimes g_{n-1}^{g_n}$$ and $$g_1^{((g_2 \otimes g_3) \otimes \ldots \otimes g_{n-1})\otimes g_n}  = g_1^{[g_2, g_3, \ldots, g_n]}.$$

Note that all involved actions are compatible. Furthermore, there is a well-defined homomorphism $\lambda^G_{n} \colon G^{\otimes n} \to \gamma_n(G)$ defined on generators by $$ ((g_1 \otimes g_2) \otimes \ldots \otimes g_{n-1}) \otimes g_{n} \mapsto [g_1, g_2, \ldots, g_n].$$ 

Here $\lambda_2^G$  corresponds to the derived map $\kappa$ of \cite{BL} and $\ker(\lambda_2^G)$ is isomorphic to $\pi_3(SK(G,1))$, where $SK(G,1)$ is the suspension of an Eilenberg-MacLane space $K(G,1)$. According to \cite[Proposition 2.8]{NR2}, the sequence $$1 \to \Delta(G) \to \ker(\lambda_2^G) \to H_2(G) \to 1$$ is exact, where $\Delta(G) = \langle g \otimes g \mid g \in G\rangle$ and $H_2(G)$ is the second homology group of the group $G$. 

In the present article we consider some finiteness conditions for the non-abelian tensor product of groups and related constructions. Finiteness conditions for the non-abelian tensor product and related constructions were considered in a number of papers (see \cite{BNR,BNR2,DLP,DLT,DM,Ellis,KS,Moravec.poly,T} and references therein for further results).
  
In \cite{DM}, Donadze and Garcia-Martinez proved that if $G$ is a finitely generated group, then the $n$-fold tensor product $G^{\otimes n}$ is finitely generated if and only if the subgroup $\gamma_n(G)$ is finitely generated. Here, we prove an analogue of this theorem for finitely presented groups.

\begin{thm} \label{thm:fp}
Let $G$ be a finitely presented group. Then, the $n$-fold tensor product $G^{\otimes n}$ is finitely presented if and only if $\gamma_n(G)$ is finitely presented.    
\end{thm}

It is straightforward to see that the finiteness of $G^{\otimes n}$ does not imply the finiteness of $G$. For instance, if $G = C_{p^{\infty}}$ is the Pr\"ufer group, then $G^{\otimes n}$ is trivial, for any $n \geqslant 2$ (see \cite{PN} for more details). The same phenomenon occurs with polycyclic groups. We describe the following finiteness conditions: 

\begin{thm} \label{thm.finite.polycyclic}
Let $G$ be a finitely generated group. 
\begin{enumerate}
    \item The $n$-fold tensor product $G^{\otimes n}$ is finite if and only if $G$ is finite.  
    \item The $n$-fold tensor product $G^{\otimes n}$ is polycyclic if and only if $G$ is polycyclic.  
\end{enumerate}     
\end{thm}

Let $G$ and $H$ be groups that act compatibly on each other. As usual,  the {\em derivative} of $G$ under the action of $H$, denoted by $[G,H]$, is defined to be the subgroup $\langle g^{-1}g^h \, \mid \, g \in G, h \in H\rangle$ of $G$. Similarly, the subgroup $[H,G] = \langle h^{-1}h^g \, \mid \, h \in H, g \in G \rangle$ of $H$ is called the derivative of $H$ under $G$. By \cite[Proposition 2.3~(b)]{BL}, there are epimorphisms: $\lambda \colon G \otimes H \to [G,H] $ and $\lambda' \colon G \otimes H \to [H,G]$ given by $\lambda(g \otimes h) =g^{-1}g^h $ and $\lambda'(g \otimes h)= h^{-g}h$, for each $g\in G$
and $h \in H$. In particular,  in the context of the $n$-fold tensor product, it is customary to write $\lambda_n^G$ rather than $\lambda'$. 

The authors of \cite{DLT} show that if $G$ and $H$ are finitely generated groups acting on each other compatibly, then the non-abelian tensor product $G \otimes H$ is finitely generated if and only if the derivatives $[G,H]$ and $[H,G]$ are finitely generated.  It is natural to ask whether the analogue of the previous result is true for finitely presented groups. We obtain the following related result.

\begin{thm}\label{thm:fpresult}
Let $G$ and $H$ be finitely generated groups acting compatibly on each other. If the derivative $[G,H]$ is finitely presented, then the non-abelian tensor product
$G\otimes H$ is finitely presented.    
\end{thm}

The paper is organized as follows. In the next section we present the proofs of Theorems \ref{thm:fp} and \ref{thm.finite.polycyclic}. Moreover, we show that if $G$ is finitely presented, then $\ker(\lambda_n^G)$ and $n$-th nilpotent multiplier $M_n(G)$ are finitely generated, for each $n \geq 2$ (see Lemma \ref{lem:mugeneralized} and Corollary \ref{cor:n-multiplier}, below). In the third section we prove Theorem \ref{thm:fpresult}. In the final section we present some Schur-Baer type theorems for finitely presented groups.

\section{$n$-fold tensor product of groups}

\noindent {\bf Convention.} Let $H$ be a normal subgroup of $G$. Then, the following homomorphisms
$G\otimes H \to [G, H]$ and $G\wedge H \to [G, H]$, $g\otimes h \mapsto [g, h]$, $g\wedge h \mapsto [g, h]$,
for each $g\in G$, $h\in H$, will be denoted by $[\;, \;]$.

Let $M,N,P$ be groups such that we can define crossed modules $\mu : M\to P$ and $\nu : N\to P$. Observe that with those maps, we can define compatible actions between $M$ and $N$. Define the fibre product as 
$$
M\times_P N=\{(m,n)\in M\times N \, | \, \mu(m)=\nu(n)\}.
$$
The (non-abelian) exterior product $M \wedge^P N$ is obtained as follows: 
$$
M \wedge^P N= \frac{M\otimes N}{\langle m\otimes n | (m,n)\in M\times_PN\rangle^{M\otimes N}}.
$$

\begin{prop}\label{prop:sequence}
Let $M,N,P$ groups such that we can define crossed modules $\mu : M\to P$ and $\nu : N\to P$.
\begin{enumerate}
    \item (Brown and Loday, \cite[Theorem 2.12]{BL}) There is a sequence: $$
\Gamma \big(M\times_PN/\langle M,N\rangle\big) \to  M\otimes N \to
 M\wedge^P N \to 1,
$$
where $\Gamma$ is Whitehead's universal quadratic functor defined in \cite{White} and $\langle M,N\rangle$ is the image of $\lambda \times \lambda'$.
    \item Let $G$ be a group and $n$ a positive integer. Then $$
\Gamma \big(\gamma_n(G)/\gamma_{n+1}(G)\big) \to \ker \Big( [\;,\;] : G\otimes \gamma_n(G)\to \gamma_{n+1}(G)\Big) \to
$$
$$
\to \ker \Big( [\;,\;] : G\wedge \gamma_n(G)\to \gamma_{n+1}(G)\Big) \to 1.
$$
\end{enumerate}
\end{prop}

\begin{proof}
(2) Adjusting the first item to our case, let $inc:\gamma_n(G)\to G$ and $Id:G\to G$, then we have that,
$$
\gamma_n(G)\times_GG=\{(g,h)\in \gamma_n(G)\times G |\ g=h\}= \gamma_n(G),
$$
and for any $g\in \gamma_n(G)$, $h\in G$,
$$
\lambda \times \lambda'(g\otimes h)= ([g,h],[g,h]).
$$
Hence, $\langle\gamma_n(G),G\rangle \cong \gamma_{n+1}(G)$. Finally,
$$
\gamma_n(G)\wedge^GG=\frac{\gamma_n(G)\otimes G}{\langle g\otimes h | g=h\rangle^{\gamma_n(G)\otimes G}}=\gamma_n(G)\wedge G.
$$
By the previous item, we obtain the following sequence: 
$$
\Gamma \big(\gamma_n(G)/\gamma_{n+1}(G)\big) \to  \gamma_n(G)\otimes G \to
 \gamma_n(G)\wedge G \to 1
$$

Therefore, it follows that,
$$
\Gamma \big(\gamma_n(G)/\gamma_{n+1}(G)\big) \to \ker \Big( [\;,\;] : G\otimes \gamma_n(G)\to \gamma_{n+1}(G)\Big) \to
$$
$$
\to \ker \Big( [\;,\;] : G\wedge \gamma_n(G)\to \gamma_{n+1}(G)\Big) \to 1.
$$
\end{proof}

First we will establish that if $G$ is finitely presented, then so is the kernel $\ker(\lambda^G_{n})$. Recall that $G^{\otimes n}$ is a central extension of $\ker(\lambda^G_{n})$ by $\gamma_{n}(G)$.

\begin{lem} \label{lem:mugeneralized}
Let $G$ be a finitely presented group. Then the kernel $\ker(\lambda^G_{n})$ is finitely generated,
for each $n\geq 2$.    
\end{lem}
\begin{proof}
The proof is by induction on $n$. Assume that $n=2$. It is well known that $\ker \Big( [\;,\;] : G\wedge G \to [G, G] \Big)$ is isomorphic to  $H_2(G)$
which is finitely generated because $G$ is finitely presented \cite[14.1.5]{Rob}. On the other hand, by \cite[Proposition 2.8]{NR2} we have a short exact sequence:
$$
1 \to \Delta(G) \to  \ker \Big( \lambda^G_{2} : G\otimes G \to [G, G] \Big) \to H_2(G) \to 1,
$$
where $\Delta(G) = \langle g \otimes g \ \mid \ g \in G \rangle$. Since $G^{ab}$ is finitely generated, $\Delta(G^{ab})$
is also finitely generated \cite[Section 3]{NR2}, which implies that $\ker \Big( \lambda^G_{2} : G\otimes G \to [G, G] \Big)$ is finitely generated.

Now, assuming that the lemma holds for $\lambda^G_{n}$,
we will prove the same for $\lambda^G_{n+1}$. Consider the central extension of groups
$$
1 \to \ker \Big( \lambda^G_{n}\Big) \to G^{\otimes n} \to \gamma_n(G) \to 1. 
$$
By \cite[Proposition 9]{BJR},  
$$
\Big(\ker\Big( \lambda^G_{n}\Big) \otimes G \Big) \to G^{\otimes n} \otimes G  \to \gamma_n(G) \otimes G \to 1
$$ is an exact sequence. 
Denote the homomorphism  $ \ker\Big( \lambda^G_{n}\Big) \otimes G  \to   G^{\otimes n} \otimes G$ 
by $\sigma$, and its image  by $\sigma \Big( \ker\Big( \lambda^G_{n}\Big) \otimes G \Big)$. Then we have the following commutative
diagram with exact rows:
\[
\xymatrix{
  1 \ar[r] & \sigma \Big( \ker\Big( \lambda^G_{n}\Big) \otimes G \Big) \ar[d]_-{ } \ar[r]^-{ } & G^{\otimes n+1}
  \ar[d]_-{\lambda^G_{n+1}} \ar[r]^-{ } & \gamma_n(G) \otimes G \ar[d]^-{[ \;, \;]} \ar[r] & 1 \\
  1 \ar[r] & 1 \ar[r]_-{} & \gamma_{n+1}(G) \ar[r]_-{Id} & \gamma_{n+1}(G) \ar[r] & 1,
}
\]
which yields the following exact sequence:
$$
 \sigma \Big( \ker\Big( \lambda^G_{n}\Big) \otimes G\Big) \hookrightarrow \ker\Big( \lambda^G_{n+1}\Big) \twoheadrightarrow
\ker \Big( [\;,\;] :  \gamma_n(G) \otimes G\to \gamma_{n+1}(G)\Big).
$$
We will show that the first and the last terms in this exact sequence are finitely generated.
Since $\ker\Big( \lambda^G_{n}\Big)$ is an abelian group and acts trivially on $G$,
by \cite[Proposition 3.2]{G} we have:
$$
 \ker\Big( \lambda^G_{n}\Big) \otimes G =  \ker\Big( \lambda^G_{n}\Big) \otimes_G I(G),
$$
where $I(G)$ denotes the augmentation ideal of $G$. Since $G$ is finitely generated, $I(G)$ is a finitely
generated $G$-module. By the induction hypothesis we get that $ \ker\Big( \lambda^G_{n}\Big) \otimes_G G$
is a finitely generated group. Hence, $\sigma \Big( \ker\Big( \lambda^G_{n}\Big) \otimes G \Big)$ is
finitely generated. Moreover, by Proposition \ref{prop:sequence}, there is an exact sequence:
$$
\Gamma \big(\gamma_n(G)/\gamma_{n+1}(G)\big) \to \ker \Big( [\;,\;] : G\otimes \gamma_n(G)\to \gamma_{n+1}(G)\Big) \to
$$
$$
\to \ker \Big( [\;,\;] : G\wedge \gamma_n(G)\to \gamma_{n+1}(G)\Big) \to 1.
$$

Since $\Gamma \big(\gamma_n(G)/\gamma_{n+1}(G)\big)$ is finitely
generated, it suffices to show that  $\ker \Big( [\;,\;] : G\wedge \gamma_n(G)\to \gamma_{n+1}(G)\Big)$ is finitely
generated. For the latter, we will use the following exact sequence (see \cite[Remark 3]{E}):
$$
H_3\big(G/\gamma_n(G)\big)\to  \ker \Big( [\;,\;] : G\wedge \gamma_n(G)\to \gamma_{n+1}(G)\Big) \to H_2(G).
$$
We have already pointed out that $H_2(G)$ is finitely generated. Moreover, since $G/\gamma_n(G)$ is polycyclic,
by \cite[Corollary 5.5]{BCM} $H_3\big(G/\gamma_n(G)\big)$ is finitely generated, which completes the proof.     
\end{proof}

We are now in a position to prove Theorem \ref{thm:fp}. 

\begin{proof}[Proof of Theorem \ref{thm:fp}]
Recall that $G$ is finitely presented. We need to show that the $n$-th term of the lower central series $\gamma_n(G)$ is finitely presented if and only if $G^{\otimes n}$ is finitely presented.  

The subgroup $\ker \Big( \lambda^G_n : G^{\otimes n} \to \gamma_n(G)\Big)$ is a central subgroup of $G^{\otimes n}$ and, by Lemma \ref{lem:mugeneralized}, it is finitely presented. Consequently, the following extension
$$
1\to \ker \Big( \lambda^G_n : G^{\otimes n} \to \gamma_n(G)\Big) \to G^{\otimes n} \to \gamma_n(G)\to 1 ,
$$
implies that the $n$-fold tensor product $G^{\otimes n}$ is finitely presented if and only if $\gamma_n(G)$ is finitely presented.
\end{proof}

Let $G$ be a group and let $F$ be a free group such that $G=F/R$ for some normal subgroup $R$ of $F$. The $n$-th nilpotent multiplier $M_n(G)$ is defined by $$ M_n(G) = \dfrac{R \cap \gamma_n(F)}{\gamma_n(R,F)},$$ where $\gamma_1(R, F)=R$ and
$\gamma_{k+1}(R, F)=[\gamma_k(R, F), F]$. In particular, $M_1(G) \cong H_2(G)$ is the Schur multiplier of $G$. 

According to Hall \cite[14.1.5]{Rob} the Schur multiplier of a finitely presented group $G$, $M_1(G)$, is finitely generated. We obtain the following related result. 

\begin{cor} \label{cor:n-multiplier}
Let $G$ be a finitely presented group. Then the $n$-th nilpotent multiplier $M_n(G)$ is finitely generated for all $n\geq 1$.    
\end{cor}
\begin{proof}
If $n=1$, then $M_1(G)$ is finitely generated \cite[14.1.5]{Rob}. Now,  we can assume $n\geq 2$. 
By \cite{BE}, $M_n(G)$ is an epimorphic image of $\ker \Big( \lambda^G_{n+1}: G^{\otimes n+1} \to \gamma_{n+1}(G)\Big)$. By Lemma \ref{lem:mugeneralized}, $M_n(G)$ is (abelian) finitely generated. 
\end{proof}

Let $\nu(G)$ be the group defined in \cite{NR1} as \[ \nu (G):= \langle 
G \cup G^{\varphi} \ \vert \ [g_1,{g_2}^{\varphi}]^{g_3}=[{g_1}^{g_3},({g_2}^{g_3})^{\varphi}]=[g_1,{g_2}^{\varphi}]^{{g_3}^{\varphi}},
\; \ g_i \in G \rangle .\]

The motivation for studying $\nu(G)$ is the commutator connection: indeed, the map  $\Phi: G \otimes G \rightarrow [G, G^{\varphi}]$,
defined by $g \otimes h \mapsto [g , h^{\varphi}]$, for all $g, h \in G$, is an isomorphism  \cite[Proposition 2.6]{NR1} (see also Ellis and Leonard \cite{EL})).

The following result is an immediate consequence of Lemma \ref{lem:mugeneralized}. 

\begin{cor}
Let $G$ be a finitely presented group. Assume that the derived subgroup $G'$ is finitely presented. Then $\nu(G)$ and $\nu(G)'$ are finitely presented.  
\end{cor}

\begin{proof}
By \cite{NR2}, there are short exact sequences:
\[
1 \to \ker(\lambda_2^G) \to \nu(G)' \to G' \times G' \times G' \to 1 
\]  
and
\[
1 \to G \otimes G \to \nu(G) \to G \times G \to 1. 
\]  
Applying Lemma \ref{lem:mugeneralized} and Theorem \ref{thm:fp} with $n=2$, we obtain that $\ker(\lambda_2^G)$ and $G \otimes G$ are finitely presented. Consequently, $\nu(G)'$ and $\nu(G)$ are finitely presented. 
\end{proof}

\begin{rem}
It is worthy to mention that part of the previous result is already known. More precisely, Kochloukova and Sidki proved that $G$ is finitely presented if and only if $\nu(G)$ is finitely presented \cite[Theorem E]{KS}.  
\end{rem}

Now we will deal with Theorem \ref{thm.finite.polycyclic}: {\it Let $G$ be a finitely generated.
\begin{enumerate}
    \item The $n$-fold tensor product $G^{\otimes n}$ is finite if and only if $G$ is finite.  
    \item The $n$-fold tensor product $G^{\otimes n}$ is polycyclic if and only if $G$ is polycyclic.  
\end{enumerate}  
}

\begin{proof}[Proof of Theorem \ref{thm.finite.polycyclic}]
First we consider the item (1). Assume that $G$ is finite. Combining \cite{Ellis} and \cite[Proposition 5]{BJR}, we deduce that $G^{\otimes n}$ is finite. 

Conversely, assume that $G^{\otimes n}$ is finite. First we prove that $G^{ab}$ is finite. As $G$ is finitely generated, we deduce that $G^{ab}$ is finitely generated and  $G^{ab} \cong T \times F,$ where $T$ is the torsion part and $F$ the free part of $G^{ab}$. In particular, $T$ is finite. There exists an epimorphism  $G^{\otimes n} \to (G^{ab})^{\otimes n}$. Moreover, by \cite{BL}, $(G^{ab})^{\otimes n}\cong (G^{ab})^{\otimes_{\mathbb{Z}} n}$, where ``$\otimes_{\mathbb{Z}}$'' denotes the usual tensor product of $\mathbb{Z}$-modules. Now, if $F$ is infinite, then so is
$(G^{ab})^{\otimes_{\mathbb{Z}} n}$, which implies that $G^{\otimes n}$ is infinite too, a contradiction.

Now, we need to show that the derived subgroup $G'$ is finite. Since $\lambda_n^G \colon G^{\otimes n} \to \gamma_n(G)$ is an epimorphism, we deduce that $\gamma_n(G)$ is finite. Consequently, it suffices to prove that if $\gamma_k(G)$ is finite, then so is $\gamma_{k-1}(G)$, where $k$ is a positive integer with $n \leq k < 2$. To prove that $\gamma_k(G)$ is finite is enough to show that $\gamma_{k-1}(G)/\gamma_{k}(G)$ is finite. Since $\gamma_k(G)/\gamma_{k+1}(G)$ is a finitely generated abelian group, is sufficient to show that 
\begin{eqnarray*} \label{eq:commutator}
    [g_1,g_2, \ldots, g_{k-1}^m]\gamma_{k}(G) & = & [g_1,g_2, \ldots, g_{k-1}]^m\gamma_{k}(G),
\end{eqnarray*}
where $m$ is a positive integer and $g_1, g_2, \ldots g_{k-1} \in G$. To this end, we proceed by induction on $m$. The formula is obvious if $m=1$. 
Assume the formula is valid up to $m-1$; we will prove it for $m$.
\begin{eqnarray*}
[g_1, \ldots, g_{k-1}^m] & = & [g_1, \ldots, g_{k-1}][g_1, \ldots, g_{k-1}^{m-1}]^{g_{k-1}} \\
& = & [g_1, \ldots, g_{k-1}][g_1, \ldots, g_{k-1}^{m-1}][g_1, \ldots, g_{k-1}^{m-1},g_{k-1}] \\
& \equiv & [g_1, \ldots, g_{k-1}]^{m} \mod   \gamma_k(G)
\end{eqnarray*}

Let $|G^{ab}|=m$, then any $m$ power of an element of $G$ is an element of $G'$, so we have that the commutator $[g_1, \ldots, g_{k-1}]\gamma_k(G) $ has order dividing $m$. We conclude that $G'$ is finite, which completes the proof. \vspace{0,5cm}

Now we deal with item (2). Assume that $G$ is polycyclic. According to Moravec's theorem \cite{Moravec.poly}, we deduce that the $n$-fold tensor product $G^{\otimes n}$ is polycyclic. 

Conversely, assume that $G^{\otimes n}$ is polycyclic. Therefore, $\gamma_n(G)$ is polycyclic. Since polycyclic is a class closed to extension, it suffices to prove that $G/\gamma_n(G)$ is polycyclic. Note that $G/\gamma_n(G)$ is nilpotent finitely generated, in particular, polycyclic.  
\end{proof}

\begin{rem}
It is worthy to mention that the case $n=2$ in Theorem \ref{thm.finite.polycyclic}~(1) was considered in \cite[Theorem 3.1]{PN}.
\end{rem}

\section{On the non-abelian tensor product of finitely presented group}

In this section, we consider the influence of specific derivative subgroup in the structure of the non-abelian tensor product. This approach is mainly motivated by the following works: Brown, Johnson and Robertson \cite[Section 8]{BJR}, Nakaoka \cite{Nak}, Visscher \cite{V} and Donadze, Ladra and Thomas \cite[Section 3]{DLT}. 

\begin{proof}[Proof of Theorem \ref{thm:fpresult}]
Recall that $G$ and $H$ are finitely generated groups acting compatibly on each other, and the derivative $[G,H]$ is finitely presented. We need to prove that the non-abelian tensor product $G \otimes H$ is finitely presented. 

By \cite[Propostion 2.3 (b) and (d)]{BL}, the non-abelian tensor product $G \otimes H$ is a central extension of $\ker(\lambda)$ by $[G,H]$, where $\lambda \colon G \otimes H \to [G,H]$, $g \otimes h \mapsto g^{-1}g^h$, is an epimorphism. According to \cite[Proposition 5.1]{DLT}, we deduce that $G \otimes H$ is finitely generated. Thus, it suffices to prove that $\ker(\lambda)$ is finitely generated.

As $G \otimes H$ is finitely generated, we have an epimorphism $p \colon F \to G \otimes H$, where $F$ is a free group of finite ranking. Define $\lambda' \colon F \to [G,H]$ as the composition $\lambda \circ p$. Then, from the following commutative diagram:

\[
\xymatrix{
  1 \ar[r] & \ker(\lambda')\ar[d]_-{ } \ar[r]^-{ } & F
  \ar[d]_-{p} \ar[r]^-{\lambda '} & [G,H] \ar[d]^-{{Id}_{[G,H]} } \ar[r] & 1 \\
  1 \ar[r] & \ker(\lambda) \ar[r]_-{} & G\otimes H \ar[r]_-{\lambda} & [G,H] \ar[r] & 1,
}
\]
we get an epimorphism $\ker(\lambda') \to \ker(\lambda) $ induced by $p \colon F \to G \otimes H$. Moreover, as $\ker(\lambda) \leq Z(G \otimes H)$, we have an epimorphism $$
\dfrac{ker(\lambda')}{[ker(\lambda'),F]} \to \dfrac{\ker(\lambda)}{[ker(\lambda),G \otimes H]} =\ker(\lambda).
$$
Set $K = \ker(\lambda')$. Now, it remains to show that $K/(K \cap F')$ and $(K \cap F')/([K,F])$ are finitely generated. Consider the following lattice of subgroups:
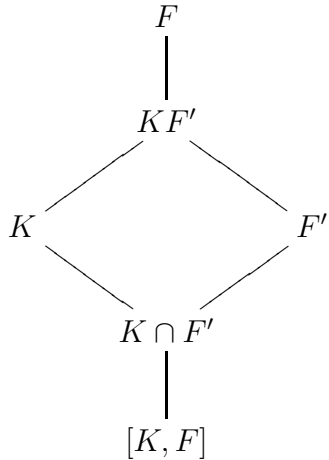
\[
\xymatrix{
& F \ar@{-}[d] & \\
& K F' \ar@{-}[dr] \ar@{-}[dl]  & \\
K \ar@{-}[dr] & & F' \ar@{-}[dl] \\
& K \cap F' \ar@{-}[d] & \\
& [K,F] & 
}
\]

Note that $K/(K \cap F')$ is isomorphic to a subgroup of the abelianization $F^{ab}$ and so, finitely generated. Moreover, by \cite[11.4.15]{Rob}, $(K \cap F')/([K,F])$ is isomorphic to the Schur multiplier $M(F/K) = M([G,H])$. According to Hall's theorem, the Schur multiplier $M([G,H])$ is finitely generated \cite[14.1.5]{Rob}. Consequently, $\ker(\lambda)$ is finitely presented, the proof is complete. 
\end{proof}

\begin{rem}
It is well-known that the finite presentability of the groups $G$ and $H$ does not imply the finite presentability of the non-abelian tensor product $G \otimes H$. For instance, by \cite[Proposition 6]{BJR}, if $F$ is a free group of rank $n \geq 2$, then the non-abelian tensor square $F \otimes F \cong F' \times \mathbb{Z}^{\frac{n(n+1)}{2}},$ where $F'$ is a free group of countably infinite rank. 
\end{rem}

\begin{thm}
Let $G$ and $H$ be groups acting compatibly on each other. Assume that $G$ is finite and $H$ finitely generated. Then the non-abelian tensor product $G \otimes H$ is finitely presented. Moreover, every subgroup of $G\otimes H$ is finitely presented. 
\end{thm}

\begin{proof}
By Theorem \ref{thm:fpresult}, the non-abelian tensor product $G \otimes H$ is finitely presented. Now, we will show that every subgroup of $G \otimes H$ is finitely presented. Indeed, consider the following
exact sequence: $$ 1 \to \ker(\lambda) \to G \otimes H \to [G,H] \to 1,$$ where $[G,H] \leq G$ is the derivative of $G$ under the action of $H$. As $G \otimes H$ is finitely generated, $\ker(\lambda) \leq Z(G \otimes H)$ and the index $[G \otimes H : \ker(\lambda)] \leq |[G,H]|< \infty$, we have the kernel $\ker(\lambda)$ is finitely generated. Therefore $G \otimes H$ is polycyclic-by-finite and the result follows.
\end{proof}

\section{Schur-Baer type theorem
for finitely presented groups}

In this section, we will use the approach of \cite{DLP} and \cite{DM} to prove a Schur-Baer type theorem
for finitely presented groups.

Recall that the group class $\mathfrak{X}$ is called a Schur class if for any group $G$ such that the factor $G/Z(G)$ belongs to $\mathfrak{X}$, also the derived subgroup $G'$ of $G$ belongs to $\mathfrak{X}$. Thus the famous Schur's theorem just states that finite groups form a Schur class \cite[10.1.4]{Rob}. Schur's theorem admits a generalization to higher terms of the lower central series. More precisely, if $G/Z_i(G)$ is finite, then $\gamma_{i+1}(G)$ is finite (Baer, \cite[14.5.1]{Rob}). In particular, the case $i=1$ is, exactly, Schur's theorem. For more details see \cite{Rob} and the references given there.    

Recall that an exact sequence of groups $$1 \to N \to H \to G \to 1$$ is called $n$-central if $N \leq Z_n(H)$. The next lemma is taken from \cite{DM}.

\begin{lem}\label{lemmaGuram}
Let $1\to N \to H \to G \to 1$ be a $n$-central extension for a fixed positive integer $n$. Then, there exists an epimorphism $\tau \colon G^{\otimes n+1} \to \gamma_{n+1}(H)$ making the following diagram commutative 
$$
\xymatrix{
    G^{\otimes n+1} \ar[rr]^{Id} \ar[dd]_{\tau} && G^{\otimes n+1} \ar[dd]^{\lambda^G_{n+1}}\\
    & \\
    \gamma_{n+1}(H) \ar[rr] && \gamma_{n+1}(G)
}
$$

\end{lem}

\begin{theorem}\label{theorem2} Let $G$ be a finitely presented group and $n\geq 1$. Then, the following are equivalent:

(i) $\gamma_{n+1}(G)$ is finitely presented,

(ii) for arbitrary $n$-central extension of groups $1 \to N \to H \to G \to 1$,
$\gamma_{n+1}(H)$ is finitely presented.
\end{theorem}
\begin{proof} Let $H$ be as in the theorem.
By Lemma \ref{lemmaGuram} there exists an epimorphism $\tau: G^{\otimes n+1} \to \gamma_{n+1}(H)$ such that the following
diagram commutes:
$$
\xymatrix{
    G^{\otimes n+1} \ar[rr]^{Id} \ar[dd]_{\tau} && G^{\otimes n+1} \ar[dd]^{\lambda^G_{n+1}}\\
    & \\
    \gamma_{n+1}(H) \ar[rr] && \gamma_{n+1}(G)
}
$$
From this diagram we get an epimorphism:
$$
\ker\Big( \lambda^G_{n+1} : G^{\otimes n+1} \to \gamma_{n+1}(G) \Big) \to
\ker\Big(\gamma_{n+1}(H) \to \gamma_{n+1}(G) \Big).
$$
Since $\ker\Big( \lambda^G_{n+1} \Big)$ is an abelian group, we get that
$\ker\Big(\gamma_{n+1}(H) \to \gamma_{n+1}(G) \Big)$ is finitely presented by Lemma \ref{lem:mugeneralized}. Consequently,
the extension of groups
$$
1\to \ker\Big(\gamma_{n+1}(H) \to \gamma_{n+1}(G) \Big) \to \gamma_{n+1}(H) \to \gamma_{n+1}(G) \to 1
$$
implies that $\gamma_{n+1}(H)$ is finitely presented if and only if $\gamma_{n+1}(G)$ is finitely presented.
\end{proof}

\begin{prop}\label{proposition1} Let $\mathfrak{X}$ be one of the following classes of groups:

\begin{enumerate}
    \item[(i)] the class of perfect finitely presented groups;
    \item[(ii)] the class of nilpotent finitely presented groups;
    \item[(iii)] the class of finitely presented groups in which every subgroup is finitely presented.
\end{enumerate}
\noindent Let $G\in \mathfrak{X}$. Then, for each $n$-central extension of groups $1 \to N \to H \to G \to 1$, $n\geq 1$,
we have $\gamma_{n+1}(H)\in \mathfrak{X} $.
\end{prop}
\begin{proof} 
We denote by $\tau : G^{\otimes n+1} \to \gamma_{n+1}(H)$ the first vertical homomorphism
in the diagram given in Lemma \ref{lemmaGuram}.

(i): Since $\gamma_{n+1}(G)=G$, by Theorem \ref{theorem2} $\gamma_{n+1}(H)$ is finitely presented. On the other hand, given two groups $M$ and $N$
acting on each other compatibly, we have the following relation (see \cite[Proposition 2.3]{BL}):
$$
[m\otimes n, m' \otimes n']= (m ^{n}m^{-1})\otimes (^{m'}n'n'^{-1}),
$$
for each $m, m'\in M$ and $n, n'\in N$. This implies that $G^{\otimes n+1}$ is a perfect group, because $G$ is perfect.
Since $\gamma_{n+1}(H)= Im (\tau)$, we get that $\gamma_{n+1}(H)\in \mathfrak{X}$.

(ii): By \cite[Theorem 3.4~(i)]{V}, the $n$-fold tensor product $G^{\otimes n+1}$ is nilpotent. Hence, $\gamma_{n+1}(H)$ is nilpotent. Moreover,
since $G$ is finitely presented and nilpotent, $\gamma_k(G)$ will be finitely presented for all $k\geq 1$.
Thus, by Theorem \ref{theorem2} $\gamma_{n+1}(H)$ is a finitely presented group.

(iii): Let $H_1$ be a subgroup of $\gamma_{n+1}(H)$. We have to show that $H_1$ is finitely presented. Let
$\delta : H \to G$ denote the epimorphism given in the proposition. Since $\delta(H_1)$ is a subgroup of $G$,
it is a finitely presented group. Therefore, it suffices to show that $\ker \Big( \delta |_{H_1} : H_1 \to \delta (H_1)\Big)$
is finitely presented. We have the following commutative diagram:
$$
\xymatrix{
    \tau^{-1}(H_1) \ar[rr]^{\lambda^G_{n+1} |_{\tau^{-1}(H_1)}} \ar[dd]^{\tau |_{\tau^{-1}(H_1)}} && \delta(H_1) \ar[dd]^{Id}\\
    & \\
    H_1 \ar[rr]^{\delta |_{H_1}} && \delta(H_1)
}
$$
This implies an epimorphism $\ker \Big(\lambda^G_{n+1} |_{\tau^{-1}(H_1)}\Big) \to \ker \Big(\delta |_{H_1}\Big)$. Since
$\ker \Big(\lambda^G_{n+1} |_{\tau^{-1}(H_1)}\Big) \subseteq \ker \Big(\lambda^G_{n+1}\Big)$, by Lemma \ref{lem:mugeneralized}
$\ker \Big(\lambda^G_{n+1} |_{\tau^{-1}(H_1)}\Big)$ is a finitely generated abelian group. Hence, $\ker \Big(\delta |_{H_1}\Big)$
is finitely presented. We are done with the proof.
\end{proof}

\begin{rem} Let $  \mathfrak{X} $ be as in Proposition \ref{proposition1}. Then $  \mathfrak{X}  $ is a Schur class.
\end{rem}

\begin{cor} Let $\mathfrak{X}$ be a class of groups defined by
$$
\mathfrak{X} = \{ G \: | \: G \:\text{is nilpotent-by-finite and finitely presented} \}.
$$
\begin{enumerate}
    \item For arbitrary two groups $G, H \in \mathfrak{X}$ acting on each other compatibly, we have $G\otimes H \in \mathfrak{X}$.
    \item Then $\mathfrak{X}$ is a Schur class.
\end{enumerate}

\end{cor}
\begin{proof} (1) By \cite[Corollary 3.6~(ii)]{DLT}, the non-abelian tensor product of two nilpotent-by-finite groups is nilpotent-by-finite. Now, it suffices to show that $G\otimes H$ is finitely presented for each $G, H\in \mathfrak{X}$. Since $G$ is finitely generated nilpotent-by-finite, it follows that $G$ and $[G,H]$ are polycyclic-by-finite and so, finitely presented.  By Theorem \ref{thm:fpresult}, the non-abelian tensor product $G \otimes H$ is finitely presented.  

(2) This follows from Lemma \ref{lemmaGuram}, Theorem \ref{theorem2}. 
\end{proof}

\section*{Acknowledgements}

 The authors are very grateful to Guram Donadze and Nora\'i Rocco for the interesting discussions and suggestions on the best approach to these results. The work of the first was supported by FAPDF and CNPq--Brazil. The second author was supported by CAPES--Brazil. 

 \section*{Data availability}

NA

\end{document}